\documentclass[12pt,english]{article}
\usepackage[utf8]{inputenc}
\usepackage[T1]{fontenc}
\usepackage{lmodern}
\usepackage[a4paper]{geometry}
\usepackage{amsmath}
\usepackage{amsfonts}
\usepackage{amssymb}
\usepackage{amsthm}
\usepackage{graphicx}
\usepackage{color}
\usepackage{comment}
\usepackage{xcolor}
\usepackage{mathabx}

\usepackage{hyperref}

\usepackage{enumitem}

\usepackage{babel}

\newtheorem{theorem}{Theorem}
\newtheorem{proposition}[theorem]{Proposition}
\newtheorem{lemma}[theorem]{Lemma}
\newtheorem{corollary}[theorem]{Corollary}

\newtheorem{remark}[theorem]{Remark}

\renewcommand*{\P}{\mathbb{P}}

\DeclareMathOperator{\Var}{Var}

\newcommand*{\ensembledenombres}{\mathbb}
\newcommand*{\R}{\ensembledenombres{R}}
\newcommand*{\C}{\ensembledenombres{C}}
\newcommand*{\N}{\ensembledenombres{N}}
\newcommand*{\RR}{\ensembledenombres{R}}

\title{On the Wasserstein distance between a hyperuniform point process and its mean}
\author{Raphael Butez\footnote{Univ. Lille, CNRS, Inria, UMR 8524 - Laboratoire Paul Painlevé, F-59000 Lille, France.}, Sandrine Dallaporta\footnote{Université de Poitiers, CNRS, LMA, Poitiers, France.}, David Garc\'ia-Zelada\footnote{Sorbonne Université, CNRS, Laboratoire de Probabilités, Statistique et Modélisation, LPSM, F-75005 Paris, France.}}

\begin{document}

\maketitle

\begin{abstract}
We study the existence of bounds on the expected $p$-Wasserstein distance between a random measure and its mean under the assumption that the $p$-th centered moments of the counting statistics are controlled uniformly in space. The average Wasserstein transport cost is shown to be bounded from above and from below by some multiples of the number of points. $D$-dimensional versions of those results are also obtained. As a corollary, we prove that for any value of $p\geq 1$ the Ginibre point process can be seen as a perturbed lattice with identically distributed perturbations with a finite $p$-th moment.
\end{abstract}

\section{Introduction}
Evaluating the transport cost between a point process and its mean, or between two independent copies of the same point process, are two very closely related questions with a long and rich history. The notion of transport
cost we consider is the Wasserstein distance
(Kantorovich-Wasserstein distance) $W_p$
whose definition is recalled in the next section.
One of the seminal papers on these questions is the work of \cite{AKT84} in the unit square whose
generalization
for $d$-dimensional cubes
says that, for a positive integer $d$ and
for a fixed $p\geq 1$,
there exists a constant $C_{d,p} > 0$ such that,
if $X_1,\dots,X_N$ is a sample of independent
uniformly distributed random variables on $[0,1]^d$
and if $\mu_N = \frac{1}{N}\sum_{i=1}^N \delta_{X_i}$
is its empirical measure and $\mu$ is the uniform
measure on $[0,1]^d$, then
\begin{equation} \label{eq:iid-macro}
	\mathbb{E}[W_p^p(\mu_N,\mu)]^{1/p} \leq
	C_{d,p}
	\begin{cases}
		\displaystyle N^{-1/2} &\text{ if } d = 1 \\
		\displaystyle
		N^{-1/2} \sqrt{\log N}    &\text{ if }
		d = 2\\
		\displaystyle
		N^{-1/d}   &\text{ if } d \geq 3 \end{cases}
\end{equation}
and the order is optimal, i.e., there exists a matching lower bound.
This statement can be found in the work of \cite[Equation (7)]{Ledoux19}. The same holds for uniformly distributed independent samples
on compact Riemannian manifolds \cite[Section 4]{Ledoux19}.
Generalizations to other i.i.d.\
samples and cost functions include
\cite{Talagrand, HighDimensions,PolishSpaces, fournierguillin,Ledoux19}.
The first terms of an asymptotic expansion have been obtained in
\cite{ambrosiostatrevisan,goldmantrevisan,ambrosiogoldmantrevisan,
	goldmanhuesmannotto}.

Let us focus on the two-dimensional case. It may be shown that an empirical measure $\mu_N$ will always be at $p$-Wasserstein distance at least $O(N^{-1/2})$ from the Lebesgue measure (see Lemma \ref{lem:lowerbound} and the appendix for more details). By decomposing the unit square as a square grid (with small squares of area $N^{-1}$) and considering
a point at each vertex,
we obtain this optimal distance to the uniform measure: $O(N^{-1/2})$.  It is quite natural to consider this square grid decomposition as a typical ``uniform'' set of points. From this point of view,
the above paragraph tells us that
uniformly distributed i.i.d.\ points are not
so ``uniform'' after all, in the sense that the average $p$-Wasserstein distance between their empirical measure and the uniform measure is larger than what is expected for a ``uniform'' grid.
It seems therefore that some interaction between
the points is needed for them to be closer
to this ``uniform'' grid.

Can we find conditions on two-dimensional point processes for which the Wasserstein distance is of the optimal order?
When the points are not independent and identically distributed, few results
about convergence rates are available. To our knowledge, the first result available is from \cite{meckesmeckes} who obtained bounds on the expected $W_p$ distance for the empirical measure of the Ginibre ensemble.
In the case of two-dimensional Coulomb gases,
the proof of the concentration inequalities from \cite{CHM18} combined with
the asymptotic expansion of the partition function stated
in the work of \cite{SandierSerfaty} give a bound
\begin{equation}
	\label{eq:optimalrate}
	\mathbb{E}[W_1(\mu_N,\mu)]  \leq CN^{-1/2}
\end{equation}
where $\mu_N$ is the empirical measure of the Coulomb gas and $\mu$ its equilibrium measure. The same bound was obtained
by \cite{carrollmarzomassanedaortegacerda} for a class of $\beta$-ensembles on compact manifolds. In his PhD thesis, \cite{prodhomme} obtained a similar result for the $W_2$ distance for the Ginibre ensemble.
Furthermore, \cite{jalowy2023wasserstein} gave
general bounds for the $W_p$ distance with a family of random matrices
improving the bounds given by
\cite{RandomMatrices}.

In this paper, we provide a condition on the centered $p$-th moment of the number of points in squares so that the optimal rate
is attained for $W_p$:
\[\mathbb{E}[W_p(\mu_N,\overline{\mu}_N)^p]^{1/p} \leqslant C_pN^{-1/2},\]
where $\overline{\mu}_N$ is $\mathbb{E}\mu_N$ up to some normalization factor, in order to deal with the fact that the number of points may be random. Similar results in dimension $d$ are also provided. This condition on moments may be linked to the notion of hyperuniformity, which is a way for a point process to encode a more regular pattern in the repartition of points than in the independent case. The proof strategy is very close to the approach of \cite{prodhomme} for the Ginibre ensemble. It relies on a multi-scale argument which was already used on this problem by several authors, including \cite{AKT84}.

Recently, \cite{lachiezereyyogeshwaran} also studied the transportation problem for hyperuniform point processes. They obtain optimal transport cost for $p$-Wasserstein distances for translation invariant point processes with some conditions on the pair correlation function and a control of the integrable reduced pair correlation function.
This project was carried out independently
of ours and, despite the overlap of the results, uses totally different techniques. In addition, \cite{huesmannleble} obtained the almost equivalence between hyperuniformity and the existence of Wasserstein bounds in dimension $2$, at the cost of a much more technical approach. The links between these works are discussed more precisely in Section \ref{sec:Comments}.

The paper is organized as follows. Section \ref{sec:Background} provides the relevant background about point processes, Wasserstein distances and hyperuniformity. Section \ref{sec:Results} first states a result on random measures (Theorem \ref{thm:minimal-hypotheses}) and then provides a version for point processes (see Corollary \ref{cor:main}), particularly determinantal ones (see Proposition \ref{prop:DPP}). Section \ref{sec:App} is devoted to applying the preceding results to the Ginibre ensemble and the Poisson point process, showing that they may be seen as perturbed lattices. After several comments and perspectives (Section \ref{sec:Comments}), the proofs are presented in Sections \ref{sec:Proofs} and \ref{sec:ProofGinibre}.

\section{Background}\label{sec:Background}

\subsection{Space of measures and point processes}

The information of a configuration of points on $\mathbb R^d$ will be captured
by its counting measure. More precisely, the information of
a locally finite multiset $X \subset \mathbb R^d$
is contained in
the locally finite measure $\mu=\sum_{x \in X} \delta_x$
and the observables
of interest are the number of points
in sets $B$ which can be nicely described as $\mu(B)$.
This motivates the following definition.

Consider the set $\mathcal M$ of locally finite measures
on $\mathbb R^d$
and define, for every bounded measurable set
$B \subset \mathbb R^d$,
the function $\pi_B: \mathcal M \to [0,\infty)$
by $\pi_B(\mu) = \mu(B)$.
We endow the set $\mathcal M$ with the $\sigma$-algebra
generated by
the functions $\pi_B$ for every bounded measurable $B$.
Moreover, the set of point configurations $\mathcal C$ is
defined as the measurable set
\[\mathcal C = \{X \in \mathcal M: X(B) \in \mathbb N
\mbox{ for every bounded measurable } B \}.\]
In this way, the map that
associates a locally finite multiset $\tilde X \subset \mathbb R^d$
to the point configuration
$X=\sum_{x \in \tilde X} \delta_x \in \mathcal C$
is a bijection.
For a measurable subset $A \subset \mathbb R^d$,
the set of locally finite measures
$\mathcal M_A$ on $A$ can be seen as a measurable subset
of $\mathcal M$ while the set of
point configurations $\mathcal C_A$ on $A$ can be seen
as a measurable subset of $\mathcal C$.
A random measure on $A$ is
a random element of $\mathcal M_A$ and a point process
on $A$
is a random element of $\mathcal C_A$.
If $\mu$ is a random measure on $A$, we can define its expected
value $\mathbb E[\mu]$ as the measure
on $A$ that satisfies
\[\mathbb{E}[\mu](B) = \mathbb{E}[\mu(B)] \mbox{ for every measurable } B
\subset A.\]
A particular kind of point processes
used in Proposition \ref{prop:DPP}
is the Hermitian determinantal point processes.
We define here the particular case of such processes on $\mathbb R^d$
with respect to Lebesgue measure.
Let $K:\mathbb R^d \times \mathbb R^d \to \mathbb C$ be a measurable function such that $K(x,y) = \overline{K(y,x)}$
for every $x,y \in \mathbb R^d$.
We say that the point process $X$ on $\mathbb R^d$
is determinantal with kernel
$K$ if $X(\{x\}) \leq 1$ for all $x \in \mathbb R^d$ and
for every $B_1,\dots, B_k \subset \mathbb R^d$
pairwise disjoint measurable subsets
\[\mathbb E[X(B_1)\dots X(B_k)]
= \int_{B_1 \times \dots \times B_k}
\det (K(x_i,x_j)_{1 \leq i,j \leq k})
\mathrm d x_1 \dots \mathrm d x_k.\]

\subsection{Wasserstein distance} \label{sub:Wasserstein}
For $p \geq 1$ we define
the Wasserstein distance (or Kantorovich distance
as remarked, for instance, in \cite{Cha16} or as
explained in \cite[Page 8]{Vershik}) between
two finite positive measures $\mu$ and $\nu$
on $\mathbb R^d$
of the same mass by
\[W_p(\mu,\nu)
= \bigg(\inf_{\Pi} \iint_{\mathbb R^d \times \mathbb R^d}
\|x-y\|^p \mathrm d \Pi(x,y) \bigg)^{1/p}
\in [0,\infty],\]
where the infimum is taken over
all positive measures $\Pi$
on $\mathbb R^d \times \mathbb R^d$
with first marginal $\mu$ and second
marginal $\nu$
(couplings of $\mu$ and $\nu$) or, more precisely,
such that $\Pi(A \times \mathbb R^d)
=\mu(A)$ and $\Pi(\mathbb R^d \times A) = \nu(A)$
for every measurable subset
$A \subset \mathbb R^d$. Since
those $\Pi$ must have the same mass as $\mu$ and $\nu$,
the requirement of $\mu$ and $\nu$
having the same mass is essential for
this definition to make sense.
In this way, $W_p$ defines
a distance on the set of finite positive
measures with given mass.
Moreover, we may notice that $W_p$ is compatible with
the measurable structure on
the measurable set of finite measures $\mathcal M_f \subset \mathcal M$,
i.e.,
$W_p$ is a measurable
function on the measurable set
$\{(\mu,\nu) \in \mathcal M_f \times \mathcal M_f:
\mu(\mathbb R^d) = \nu(\mathbb R^d) \}$.

Our main goal in this work is to use some
hyperuniformity conditions to obtain
optimal rates for the expected value
of the Wasserstein distance.
Since the notion of hyperuniformity (defined below)
typically applies to point processes and not to
sequences of empirical measures,
we would like to translate
the result stated in the introduction
to a different scale.
To achieve this, we first notice the following.
Consider two measures $\mu$ and $\nu$ on $\mathbb R^d$
with the same mass and two positive constants
$\lambda, N >0$. Define $f:\mathbb R^d \to \mathbb R^d$
by $f(x) = \lambda x$. Then, by using that
$\Pi$ is a coupling of $\mu$ and $\nu$
if and only if $N (f \times f)_* \Pi$ is a coupling of
$N f_* \mu$ and $N f_* \nu$, we get
\[W_p^p(N f_* \mu, N f_* \nu) = N \lambda^p W_p^p(\mu,\nu).\]
With this in mind, Equation \eqref{eq:iid-macro}
would read as follows.
For a positive integer $d$ and
for a fixed real number $p\geq 1$,
there is a constant $C_{d,p} > 0$ such that,
if $X_1, \dots, X_N$ is an independent
sequence of
uniformly distributed random variables on $[0,N^{1/d}]^d$,
and if we let $\mu_N = \sum_{i=1}^N \delta_{X_i}$
and consider the Lebesgue measure $\mu$ on $[0,N^{1/d}]^d$,
\begin{equation} \label{eq:iid-micro}
	\mathbb{E}[W_p^p(\mu_N,\mu)] \leq
	C_{d,p}
	\begin{cases}
		\displaystyle N^{1+p/2} &\text{ if } d = 1 \\
		\displaystyle
		N (\log N)^{p/2}    &\text{ if }
		d = 2\\
		\displaystyle
		N   &\text{ if } d \geq 3 \end{cases}.
\end{equation}
This setting,
where the important object is a point process
(a ``counting'' measure),
can be seen as the \textit{microscopic} setting
while the rescaled version considered in the introduction,
where the empirical measure is the main object, can be thought
of as the \textit{macroscopic} setting.
We have decided
to write our main result Theorem \ref{thm:minimal-hypotheses}
in the microscopic setting in which we see a
linear upper bound
$\mathbb{E}[W_p^p(\mu_N,\mu)] \leq C N$.

A standard argument gives a lower bound for the distance between a uniform measure on a finite set and an absolutely continuous measure with respect to Lebesgue measure; see the following lemma. For the sake of completeness, we include the proof
in the appendix.
\begin{lemma}[Deterministic lower bound]
	\label{lem:lowerbound}
	There exists a constant $\alpha_d>0$ such that,
	for any integer $N \geq 1$, any
	$x_1,\dots, x_N \in \mathbb R^d$ and any
	measure $\mu$ on $\mathbb R^d$ of mass $N$ with
	density with respect to Lebesgue measure
	bounded from above by a constant $A>0$, one has,
	for every $p \geq 1$,
	\[ W_p^p\Big(\sum_{i=1}^N \delta_{x_i}, \mu\Big) \geq (\alpha_d)^p
	\frac{N}{A^{p/d}}  .\]

\end{lemma}
Note that for singular measures, \cite[Proposition 3]{dereich2013constructive} shows that this lower bound is no longer true and that faster convergence speed can be reached.

\subsection{Hyperuniformity}


In this subsection we give the standard notion of hyperuniformity. Although there is such a definition for general dimension, we focus on dimension 2. Indeed, in this case, the main hypothesis in Theorem \ref{thm:minimal-hypotheses} has a clear link with the notion of hyperuniformity, whereas it is not the notion we use for the analogous result in dimension $d$ (Theorem \ref{thm:minimal-hypotheses dimension}). We refer to \cite[Equation (97)]{Torquato}
for hyperuniformity in general dimension (see also \cite{Coste} and \cite{lachiezereySurvey}).


A hyperuniform point process on $\mathbb R^2$ is usually defined as a point process $X$
on $\mathbb R^2$ whose law is invariant under translations and for which
\[ \lim_{R \to +\infty} \frac{\mathrm{Var}(X(D(0,R)))}{|D(0,R)|} = 0 . \]
Here $D(0,R)$ denotes the disk of radius $R$ centered at $0$.
There is a vast literature on hyperuniform point processes and we refer to the surveys of \cite{Coste}, \cite{lachiezereySurvey} and  \cite{Torquato} on this topic.
The usual motivation for this definition is to think of them as
``more uniform'' or
more ordered than
stationary Poisson point processes for which
$\mathrm{Var}(X(D(0,R)))$ is proportional to $|D(0,R)|$. The Ginibre point process and the zeros of the flat Gaussian analytic function are two examples of hyperuniform point processes. Intuitively, the points of a hyperuniform point process are well spread out in space, with no cluster or empty spots,
which makes those processes easier to match to a continuous measure and good candidates to match the optimal transport cost.
Usually three classes are defined among hyperuniform point processes according to the decay rate of $\mathrm{Var}(X(D(0,R)))$.
\[ \begin{cases}
	\text{Type \,\,I \hspace{2mm} if: }
	&\mathrm{Var}(X(D(0,R))) \lesssim |D(0,R)|^{1/2}, \\
	\text{Type \,II \hspace{1mm} if: } &\mathrm{Var}(X(D(0,R))) \lesssim |D(0,R)|^{1/2} \log R, \\
	\text{Type III \ if: } &\mathrm{Var}(X(D(0,R))) \lesssim |D(0,R)|^{1-\varepsilon}.
\end{cases}\]
Here, $\lesssim$ means that the left-hand side is bounded by a constant times the right-hand side.

These categories do not cover all possible hyperuniform point processes, and one can build a point process with more or less any growth of the number variance; see \cite[Theorem 3]{dereudreflimmelhuesmannleble}. Still, those classes give an idea of what can be expected for most point processes.

In this work, we will consider squares instead of disks when taking
the variance.
There exist point processes that are hyperuniform with disks and not hyperuniform with squares, but as soon as the hyperuniformity can be ensured via a control of the structure factor at the origin, the point process is hyperuniform for any reasonable window shape, see \cite{Coste} and \cite{lachiezereySurvey}.
On the other hand, our point processes do not need
to be defined on the whole plane but only on squares, so that invariance under translations
is not really a meaningful condition here.

\section{Results} \label{sec:Results}
We begin with stating a result on random measures on dimension $2$ and then we give
its $d$-dimensional version. A corollary for
a $p$-th moment version of type III hyperuniform
point processes will be given later. Finally,
usual type III hyperuniform
Hermitian determinantal point processes are shown to satisfy our hypotheses
for every finite $p \geqslant 1$.

\begin{theorem}\label{thm:minimal-hypotheses}
	Fix $p \geq 1$ and consider a real number $N\geq 1$. Let $\mu_{N}$ be a random measure on
	the square $Q_N=[0,\sqrt{N}]^2$ such that the following conditions
	hold.
	\begin{enumerate}
		\item\label{Hypo1:moment_minimal} There exists
		a monotone function
		$g:[1,N]\to (0,\infty)$,
		either non-decreasing or non-increasing, such that for
		any square $B \subset Q_N$ with
		$|B| \geqslant 1$,
		\begin{equation}
			\mathbb{E}[|\mu_N(B)-\mathbb{E}[\mu_N(B)]|^p] \leqslant |B|^{p/2}
			g(|B|)^p.
		\end{equation}
		\item\label{Hypo3:density_minimal} There exist two constants
		$a \in (0,1]$ and $A \in [1,\infty)$
		such that for any
		square $B \subset Q_N$
		\begin{equation}
			a |B| \leqslant  \mathbb{E}[\mu_N(B)]
			\leqslant A |B| .
		\end{equation}
		This is equivalent to the requirement that $\mathbb{E}[\mu_N]$ has a density
		with respect to the Lebesgue measure bounded above
		by $A$ and below by $a$.
	\end{enumerate}
	Let $M_N = \mu_N(Q_N)$ and define the random measure
	$\overline{\mu}_{N} = \frac{M_N}{\mathbb{E}[M_N]}\mathbb E[\mu_{N}]$
	of total mass $M_N$.
	Then there exists a constant $C_p$
	(depending only on $p$) such that
	\[\mathbb{E}[W_p^p(\mu_N, \overline{\mu}_N)]  \leqslant
	(C_p a^{1-2p} A^p)
	N \Big( 1 + g(1) + g(N) + \int_1^{N} g(y) \frac{\mathrm d y}{y} \Big)^p.\]
\end{theorem}


Note that for $N$ small we can get a better upper bound using the crude upper bound
$W_p^p(\mu_N, \overline{\mu}_N) \leq
\mathrm{diam}(Q_N)^p M_N$ which yields
$\mathbb{E}[W_p^p(\mu_N, \overline{\mu}_N)] \leq
\mathrm{diam}(Q_N)^p A N
= 2^{p/2}A N^{1+p/2}$.
Indeed, the latter inequality implies that, if
$N \in (0,R)$ for some fixed $R$,
we can choose
$C_p$ larger than $(2 R)^{p/2}$
and obtain again the bound in Theorem
\ref{thm:minimal-hypotheses}.

For $N$ i.i.d.\ uniformly distributed points on $Q_N$
and $p \geq 1$ fixed, Hypothesis \ref{Hypo1:moment_minimal} is true for $g$ constant. The bound obtained
in this case is proportional to $N (\log N)^p$, which is
not the optimal bound
$N (\log N)^{p/2}$ given in Equation \eqref{eq:iid-micro}.
On the other hand, if the random measure $\mu_N$ is such that Hypothesis \ref{Hypo1:moment_minimal} is true for
$g(y) = \frac{1}{(1+\log y)^r}$
with $r > 1$, we obtain a linear bound
in that case (as $y \mapsto y^{-1} g(y)$ is integrable).


\begin{remark}
	If we are only interested in the $p=1$ case we could change
	Hypothesis \ref{Hypo1:moment_minimal} to
	\[\mathbb{E} \Big[ \Big| \frac{\mu_N(B)}{\mathbb{E}[\mu_N(B)]} - 1
	\Big|\Big]
	\leqslant
	|B|^{-1/2} g(|B|)
	\mbox{ for every square } B
	\mbox{ that satisfies } |B| \geq 1
	\]
	and eliminate Hypothesis \ref{Hypo3:density_minimal} to
	obtain a universal constant $\alpha$ such that
	\[\mathbb{E}[W_1(\mu_N, \overline{\mu}_N)]  \leqslant \alpha
	\mathbb E[\mu_N(Q_N)] \Big( 1 + g(1) + g(N) + \int_1^{N} g(y) \frac{\mathrm d y}{y} \Big).\]
\end{remark}

The previous theorem is stated in dimension $2$ which is the most studied case. Nevertheless, our proof applies in any dimension and a general theorem follows. Its proof is exactly the same up to minor modifications.

\begin{theorem}\label{thm:minimal-hypotheses dimension}
	Fix $p \geq 1$ and consider a real number $N\geq 1$. Let $\mu_{N}$ be a random measure on
	the cube $Q_N=[0,N^{1/d}]^d$ such that the following conditions
	hold.
	\begin{enumerate}
		\item There exists
		a monotone function (non-decreasing or non-increasing)
		$g:[1,N] \to (0,\infty)$ such that, for
		any cube $B \subset Q_N$ that satisfies
		$|B| \geqslant 1$, it holds that
		\begin{equation} \label{Hypo1:moment_minimal_d}
			\mathbb{E}[|\mu_N(B)-\mathbb{E}[\mu_N(B)]|^p] \leqslant |B|^{p\frac{(d-1)}{d}}
			g(|B|)^p.
		\end{equation}
		\item There exist two constants
		$a \in (0,1]$ and $A \in [1,\infty)$
		such that for any
		cube $B \subset Q_N$
		\begin{equation} \label{Hypo3:density_minimal_d}
			a |B| \leqslant  \mathbb{E}[\mu_N(B)]
			\leqslant A |B| .
		\end{equation}
	\end{enumerate}
	Let $M_N = \mu_N(Q_N)$ and
	$\overline{\mu}_{N} = \frac{M_N}{\mathbb{E}[M_N]}\mathbb E[\mu_{N}]$.
	Then there exists a constant $C_{p,d}$ such that
	\[\mathbb{E}[W_p^p(\mu_N, \overline{\mu}_N)]  \leqslant
	(C_{p,d} a^{1-2p} A^p)
	N \Big( 1 + g(1) + g(N) + \int_1^{N} g(y) \frac{\mathrm d y}{y} \Big)^p.\]
\end{theorem}

Since $aN \leq \mathbb E[\mu_N(Q_N)] \leq A N$, we should think of $N$ as being
the mass of $\mu_N(Q_N)$ or, if $\mu_N$ is a point process,
as the number of particles of $\mu_N$. As explained
in Subsection \ref{sub:Wasserstein}, we may convert
this theorem into one about measures on the unit square $Q_1$ by scaling. In this setting, for $B \subset Q_1$,
\eqref{Hypo1:moment_minimal_d} becomes
$\mathbb{E}[|\mu_N(B)-\mathbb{E}[\mu_N(B)]|^p] \leqslant N^{p \frac{(d-1)}{d}} |B|^{p\frac{(d-1)}{d}}
g(N|B|)^p$ and \eqref{Hypo3:density_minimal_d} becomes
$ 	a N|B| \leqslant  \mathbb{E}[\mu_N(B)]
\leqslant A N|B| $. Since $N$ may be thought of
as the number of particles, $\mu_N/N$ can be thought of as
the empirical measure and
\eqref{Hypo3:density_minimal_d} would be related to
bounds on the density of $\lim_{N \to \infty} \mu_N/N$, the so-called
equilibrium measure. In this way,
a hypothesis like \eqref{Hypo3:density_minimal_d} makes sense on
metric spaces with different classes of shapes.
An example of such bounds
for some sequences of point processes on manifolds can be found in
\cite[Theorem 1.6]{CHE20}.
Moreover, we may replace $Q_1$ by a
flat $d$-dimensional torus.
Theorem \ref{thm:minimal-hypotheses dimension}
still holds in this case
since the distance in the torus is smaller than the distance in the cube.

\begin{remark}[On the first hypothesis of Theorem
	\ref{thm:minimal-hypotheses dimension}]\label{remark lattice}
	At first glance, the exponents in Condition \eqref{Hypo1:moment_minimal_d} seem somewhat peculiar.
	Nevertheless, they are more illustrative
	if we think $ |B|^{\frac{(d-1)}{d}}$ as essentially being the area $|\partial B|$
	of the boundary $\partial B$ so that
	\eqref{Hypo1:moment_minimal_d}
	becomes a condition of the form $\|\mu_N(B)-\mathbb{E}[\mu_N(B)]\|_p \leqslant |\partial B| g(|B|)$.
	In some sense, this condition
	is asking to what extend the
	fluctuations of the number of particles
	happen only through its boundary.

	Instances of classical point processes
	where fluctuations happen through
	the boundary of $B$,
	i.e., where
	$\| \mu_N(B) - \mathbb E[\mu_N(B)]\|_p
	\leq C |\partial B|$
	for some constant $C$,
	are the so-called perturbed lattices. More precisely, if we consider
	a family of random variables
	$(X_z)_{z  \in \mathbb Z^d}$
	bounded by a constant $R>0$, then the point process formed by
	$(z+X_z)_{z \in \mathbb Z^d}$ (a perturbed lattice) satisfies Condition
	\eqref{Hypo1:moment_minimal_d}  with $g$ constant.
	Then, we may reinterpret
	Condition \eqref{Hypo1:moment_minimal_d}
	as asking how much the number of particles
	fluctuates ``better'' than in the
	case of perturbed lattices.

	If the perturbation variables $(X_z)_{z \in \mathbb Z^d}$
	are independent,
	then we have an even better bound
	$\| \mu_N(B) - \mathbb E[\mu_N(B)]\|_2
	\leq C\sqrt{|\partial B|} $ and this condition is sometimes considered
	as the hyperuniformity of type I in general dimension.
	An example of the latter ``type I'' kind of bound for domains
	with smooth boundary on manifolds
	can be found in \cite[Theorem 1.6]{CHE20}.
	We may also see \cite{dereudreflimmelhuesmannleble}
	for connections between hyperuniformity and perturbed lattices.
\end{remark}

A particular interesting case
is the one for a $p$-th moment
version of type III hyperuniform point process
that, since we have only defined type III for $\mathbb R^2$, is stated in this setting.

\begin{corollary}\label{cor:main}
	Fix $p\geq 1$ and let $X$ be a point process on $\RR^2$
	and let $\mathbb{E}[X]$ be its expected value. Suppose that the following
	properties hold.
	\begin{enumerate}
		\item\label{Hypo1:moments_main} There exist
		$ \varepsilon \in (0,1)$ and $\gamma>0$ such that, for
		any square $B \subset \mathbb R^2$ with $|B| \geqslant 1$,
		\begin{equation}
			\mathbb{E}\big[|X(B)-\mathbb{E}[X(B)]|^p\big] \leqslant
			\gamma |B|^{(1-\varepsilon)p/2}.
		\end{equation}
		\item\label{Hypo3:density_main} There exist $a,A>0$ such that for any square $B \subset \mathbb R^2$,
		\begin{equation}
			a |B| \leqslant  \mathbb{E}[X(B)] \leqslant A |B| .
		\end{equation}
		This is equivalent to requiring that $\mathbb E[X]$ has a density bounded from above and from below with respect to the Lebesgue measure on
		$\mathbb R^2$.
	\end{enumerate}
	Then there exists $\alpha,\tilde \alpha > 0$, such that
	for any $N \geq 1$,
	if $C_N= [-\sqrt{N}/2,\sqrt{N}/2]^2$,
	\[ \tilde \alpha N \leq \mathbb{E}\bigg[W_p^p\bigg(X|_{C_N}, \frac{X(C_N)}{\mathbb{E}[X(C_N)]} \mathbb{E}[X]|_{C_N}\bigg)\bigg]
	\leqslant \alpha N.\]
\end{corollary}

The $d$-dimensional version would ask for a bound
$    \mathbb{E}\big[|X(B)-\mathbb{E}[X(B)]|^p\big]^{1/p} \leqslant
\gamma |B|^{(1-\varepsilon) \frac{(d-1)}{d}}$, which is weaker than the usual
type III hyperuniformity found in the literature whose $p$-th moment
version would be
$\mathbb{E}\big[|X(B)-\mathbb{E}[X(B)]|^p\big]^{1/p} \leqslant
\gamma |B|^{(1-\varepsilon) /2}$.

\subsection*{The special case of determinantal point processes}
In order to apply Theorem \ref{thm:minimal-hypotheses} or \ref{thm:minimal-hypotheses dimension}, the most difficult hypothesis to check is \ref{Hypo1:moment_minimal}, which amounts to showing some kind of $p$-hyperuniformity for point processes. Showing these bounds can be a very delicate task, but fortunately, in the special case of determinantal point processes, the following Proposition states that it is sufficient to check this bound for $p=2$ to obtain it for any value of $p$. Hence, hyperuniform determinantal point processes satisfy Hypothesis \ref{Hypo1:moment_minimal} for any $p \geqslant 1$.

This comes from the fact that, for Hermitian determinantal point processes, the $L^2$-norm of the
deviation essentially bounds its $L^p$-norm, which is the exact converse of the Jensen inequality on moments. This property is sometimes called hypercontractivity.

\begin{proposition}\label{prop:DPP}
	Fix $p \geq 1$.
	There exists $\gamma_p > 0$ such that, for any
	Hermitian
	determinantal point process $X$ and any measurable set $B$
	such that $\mathbb E[X(B)]$ is finite, we have
	\[\mathbb{E}\big[|X(B)-\mathbb{E}[X(B)]|^p\big]
	\leq \gamma_p (\Var(X(B))^{p/2} + 1).\]
	In particular, a Hermitian determinantal point process $X$ on $\mathbb R^2$
	satisfying
	the first hypothesis of Corollary
	\ref{cor:main} for $p=2$
	also satisfies it for every $p \in [1,\infty)$ and the same $\varepsilon$.
\end{proposition}

\section{Application to some specific point processes}\label{sec:App}

\subsection{Perturbed lattices}
A perturbed lattice is a point process of the form $X=\{k+\xi_k, k \in \mathbb{Z}^2  \}$, where $\{ \xi_k, k \in \mathbb{Z}^2\}$ is a collection of identically distributed random variables which can be either independent, or invariant under shifts. These points processes have been intensively studied as they have strong ties with the notion of hyperuniformity. In fact, \cite{dereudreflimmelhuesmannleble} proved that in dimension $2$, the perturbed lattice is a hyperuniform point process as soon as the perturbations have a finite second moment.
\\
Theorem \ref{thm:minimal-hypotheses} applies to perturbed lattices, but the same result may be obtained by an elementary calculation, transporting the perturbed lattice towards $\mathbb{Z}^2$ with cost $\mathbb{E}(|\xi_{(0,0)}|^2)$ per point and then transporting the grid towards the Lebesgue measure.

\subsection{Ginibre ensemble}
One of the most studied hyperuniform point process is the Ginibre ensemble, which can be seen either as the limit as $N$ goes to infinity of the eigenvalue process of a $N \times N$ random matrix with
i.i.d.\ complex Gaussian coefficients, or as the determinantal point process in the plane with kernel $K(z,w) = e^{z \bar{w}-|z|^2/2-|w|^2/2}$. This point process has unit intensity, is invariant under translation and rotations, and is known to be hyperuniform of type I. The recent results \cite{lachiezereyyogeshwaran} and \cite{huesmannleble} imply that this point process can be seen as a perturbed lattice where the perturbations have a finite second moment.
The next consequence of Theorem \ref{thm:minimal-hypotheses} may be viewed
as an $L^p$ analogue of this result.

\begin{proposition}\label{Ginibre}
	Let $\mathrm{Gin}$ be the Ginibre point process on $\mathbb{C}$, defined as above. For any $p\geq 1$, there exists a family of identically distributed random variables $(\xi_k)_{k\in \mathbb{Z}^2}$ whose joint distribution is invariant under translations satisfying $\mathbb{E}(|\xi_{(0,0)}|^p) < +\infty$ such that, in law,
	\[\{ k + \xi_k, k \in \mathbb{Z}^2     \} = \mathrm{Gin}.\]
\end{proposition}

In the case $p=2$, this result is already included in \cite{lachiezereyyogeshwaran} and \cite{huesmannleble}, but their approaches cannot be extended to higher values of $p$. With the notations from \cite{huesmannleble, Huesmann2016, EHJM}, Proposition \ref{Ginibre} can be stated as $\mathrm{Wass}_p(\mathrm{Gin},\mathrm{Leb}) <+\infty$, where $\mathrm{Wass}_p$ is the Wasserstein distance between translation invariant point processes associated to the cost function $c(x,y)=|x-y|^p$. This distance between point processes will be presented in the proof of Proposition \ref{Ginibre}.
Obtaining information on the regularity of the perturbation is in general a very delicate task. To our knowledge, the only other classical point process for which such information is known is the zero process of the Flat Gaussian Analytic Functions, which is known to be a perturbed lattice with sub-Gaussian perturbations \cite{SodinTsirelsonII}. An interesting question would be to show that the Ginibre ensemble can be viewed as a sub-exponential or sub-Gaussian perturbed lattice, and we believe that the techniques presented in the present paper could be generalized to cover this setting.

\subsection{Zeros of Gaussian analytic functions.}
The flat Gaussian analytic function is usually defined by $f(z)=\sum_{n=0}^{\infty}a_n\frac{z^n}{\sqrt{n!}}$ for $z \in \C$ and with $(a_n)_{n \geq 0}$ a family of i.i.d. standard complex Gaussian random variables. The zeros of this random entire function are one of the most interesting example of a type I hyperuniform point process and Theorem \ref{thm:minimal-hypotheses} applies in this case. This allows to recover the fact that this point process is a $L^2$ perturbed lattice. Yet, we do not recommend using our paper on this point process as
 \cite{SodinTsirelsonII} show that it is hyperuniform as a byproduct of the fact that it is a perturbed lattice, with sub-gaussian perturbations $\xi_k$'s.

\subsection{Poisson point process with $d\geq 3$.}
In dimension $d\geq 3$ and with $p=2$, one can apply Theorem \ref{thm:minimal-hypotheses dimension}  to the empirical measure of $N$ i.i.d.\ points, uniformly distributed on the $d$-cube $Q_N$. We may choose
$g(x)=x^{\frac{1}{d} - \frac{1}{2}}$ and we recover the fact that the Poisson point process is at finite $\mathrm{Wass}_2$ distance from the Lebesgue measure in dimension greater than $3$. This implies that the Poisson point process can be viewed as an $L^2$-perturbed lattice, using the same tools as in the Ginibre case.

\subsection{Hyperuniform point processes}
In dimension $2$, if $X$ is a type I, II or III translation invariant hyperuniform point process, one can apply Theorem \ref{thm:minimal-hypotheses} for $p=2$, where the first hypothesis is satisfied with the function $g(x)=x^{-\delta}$ for some $\delta >0$ depending on the hyperuniformity type. This recovers the recent results \cite{lachiezereyyogeshwaran} and \cite{huesmannleble} on these point processes. Following \cite{huesmannleble}, this implies that $X$ has the same distribution as $\{k + \xi_k, k \in \mathbb{Z}^2  \}$, where $\{ \xi_k, k \in \mathbb{Z}^2\}$ is a shift invariant family of $L^2$ random variables.
A lot of progress has been made on hyperuniform point processes since our paper was first written. The most important recent contribution being the type I hyperuniformity of the infinite Coulomb gas at any temperature, recently obtained by \cite{lachiezereyHU}. Hence the infinite Coulomb gas with inverse temperature $\beta>0$ is a $L^2$ perturbed lattice.
To obtain the same result as Proposition \ref{Ginibre} for any $\beta>0$, one would need to show a Bernstein-type inequality for the number of points of the Coulomb gas in boxes. To our knowledge, this result is not available but may not be out of reach: \cite{Leble2} showed sub-Gaussian behavior at small scales but some sub-exponential behavior for higher exponential moments would be required and are not available at the moment.

\section{Related results and comments}\label{sec:Comments}
Two papers \cite{huesmannleble} and \cite{lachiezereyyogeshwaran} were done independently and simultaneously to this paper and give very close results. We refer to the recent survey \cite{lachiezereySurvey} for a unified presentation of the results and some elements of comparison.

In \cite{huesmannleble}, the authors obtain a full description in dimension $2$ of the connection between hyperuniformity, the finiteness of the Coulomb energy of the point process, and transport properties. One of their results is to prove that star-hyperuniformity, which is very close to Hypothesis \ref{Hypo1:moment_minimal} from Theorem \ref{thm:minimal-hypotheses}, implies a control of the transport cost by $N(1+ \int_1^{+\infty} g(y)^2/y dy)$, which is better than the bound involving $g(y)/y$ that we obtain. These two bounds give different results in the case where $g(y)=\log(y)^{-a}$, with $a\in (1/2,1)$.
In \cite[Theorem 3]{lachiezereyyogeshwaran}, the authors obtained a theorem which deals with the same question, but with a very different approach
and in the setting of stationary point processes.
For a precise discussion on the relation between the different hypotheses of our paper, and the ones of \cite{lachiezereyyogeshwaran} and \cite{huesmannleble}, we refer to \cite[Proposition 1.7]{huesmannleble}.

The two papers \cite{huesmannleble} and \cite{dereudreflimmelhuesmannleble}
have, very recently, provided a different perspective on the main result
in the setting of stationary point processes. On the one hand, Theorems 1 and 4 of \cite{huesmannleble} together say that a condition which is slightly weaker than the hypothesis of
Theorem \ref{thm:minimal-hypotheses}  for $p=2$ but stronger than classical hyperuniformity implies $2$-Wasserstein linear (in $N$) bounds. On the other hand,
we may find in \cite[Theorem 1]{dereudreflimmelhuesmannleble}
that an $L^2$-perturbed lattice in dimension $2$ is hyperuniform
which can be restated, by using an adapted version  of \cite[Proposition 2]{lachiezereyyogeshwaran},
as saying that
linear $2$-Wasserstein bounds imply hyperuniformity.
These implications nearly give the full picture for the relation between hyperuniformity and $2$-Wasserstein linear bounds
for $2$-dimensional stationary point processes.


In dimension greater than $2$, there exist perturbed lattices which are not hyperuniform with the usual definition of hyperuniformity
given by comparing the process with a stationary Poisson point process. The connection between hyperuniformity, Wasserstein bounds and perturbed lattices is now almost clear in dimension 2, but our result
tells us that the more pertinent bound
is the one obtained by comparing the process with a perturbed lattice
as explained in Remark \ref{remark lattice} above. This is an interesting direction for further investigations.

To conclude, this paper has two main assets. First, we provide bounds which are valid for general random measures, not only stationary point processes, any dimension and any finite $p \geq 1$. Second, the proof is short and not very technical.
One direction of improvement for the main result of this paper would be to relax the shape dependency for the moments estimate, i.e., to consider, for instance, disks instead of squares in the first hypotheses of
Theorem \ref{thm:minimal-hypotheses}, or to replace $Q_N$ by other shapes or more general metric spaces.

\section*{Acknowledgments}
We thank Michaël Goldman, Jonas Jalowy, Raphaël Lachièze-Rey, Clément Sarrazin and Thomas Leblé for useful discussions and insights. This work was supported by the CNRS via the PEPS funding ``\emph{Vitesse de convergence de mesures spectrales empiriques}''. We are also glad to acknowledge support from
the GdR \emph{Matrices et Graphes Aléatoires} (now included in the RT \emph{Mathématiques et Physique}) which made this collaboration possible.

\section{Proofs}\label{sec:Proofs}
We prove all results of Section \ref{sec:Results}
in the same order as they are stated.
\subsection{Proof of Theorem \ref{thm:minimal-hypotheses}}


The main idea is to interpolate between $\mu_N$ and $\overline{\mu}_N$ by starting with $\overline{\mu}_N$ on $[0,\sqrt N]^2$ and dividing the square in four equal squares. Consider a measure $\nu_1$
that on each of those new squares $B$ is
proportional to
${\overline{\mu}_N}|_B$ but
whose mass is $\mu_N(B)$. Then, divide each of the new squares in four equal squares $B'$
and consider again a measure $\nu_2$ proportional to
${\overline{\mu}_N}|_{B'}$ but
whose mass is $\mu_N(B')$. Repeat this procedure until some $K$-th subdivision where the smallest squares
have size of order $1$. Now, the point is
that to compare $\nu_{k}$ with $\nu_{k+1}$ we may restrict ourselves
to the squares in the $k$-th subdivision, and our task becomes a comparison between
some measure and a ``subdivision'' of it in four equal squares.
The key is \cite[Lemma 3.4]{goldmantrevisan}
which we provide here under the name of Lemma \ref{lem:goldmantrevisan}.
Finally, we compare the $K$-th step measure $\nu_K$ with the measure $\mu_N$
by comparing again $\nu_K$ and $\mu_N$ on each
of the squares of the $K$-th subdivision
but now using that
the smallest squares
have size of order $1$.

For future use,
notice first that, for any square $B \subset [0,\sqrt N]^2$,
\begin{equation}
	\label{eq:Hyp4}
	\mathbb{E} \Big[ \Big| \frac{\mu_N(B)}{\mathbb{E}[\mu_N(B)]} - 1
	\Big|^p\Big]
	=\frac{1}{\mathbb{E}[\mu_N(B)]^p}
	\mathbb{E}
	\Big[ \Big| \mu_N(B)- \mathbb{E}[\mu_N(B)]
	\Big|^p\Big]
	\leqslant
	\frac{1}{a^p} |B|^{-p/2} g(|B|)^p
\end{equation}
provided that $|B|\geq 1$. We shall suppose $g$ non-increasing,
the non-decreasing case works in a similar way.


\vspace{3mm}
\noindent
\textbf{Step 1 -- Interpolation:}
We start by defining a finite sequence of intermediate measures which interpolate between
$\overline{\mu}_N$ and $\mu_N$.
Let us denote $K = \lfloor \log_4(N) \rfloor = \lfloor \log_2(\sqrt N) \rfloor \in \N$. Let us consider, for each $k \in \{0, \dots, K  \}$,
the set  $\mathcal B_k$ of $4^k$ squares which form a partition of  $[0,\sqrt{N})^2$ obtained by recursively dividing each square into 4 equal squares.
Equivalently,
let us partition $[0,\sqrt N)$ in $2^k$ intervals
of length $\sqrt N 2^{-k}$
as
\[\mathcal I_k = \Big\{\Big[\frac{\sqrt N(i-1)}{2^k}, \frac{\sqrt N i}{2^k}\Big): i \in \{0,\dots, 2^k-1\} \Big\}\]
and use this to partition the square $[0,\sqrt{N})^2$
in the $4^k$ squares given by products of intervals
from $\mathcal I_k$, i.e.,
$\mathcal B_k = \{ I_1 \times I_2 : I_1, I_2 \in \mathcal I_k\}$. Notice that the smallest
square obtained in this way has length
$\sqrt{N} 2^{-K} \geq
\sqrt{N} 2^{-\log_2(\sqrt N)} = 1$.
For any $k \in \{1, \dots K\}$, we define the measure
\[  \nu_k = \sum_{B \in \mathcal B_k} \frac{\mu_N(B)}{\overline{\mu}_N(B)}
{\overline{\mu}_N}|_B = \sum_{B \in \mathcal B_k} \frac{\mu_N(B)}{\mathbb{E}[\mu_N(B)]}
\mathbb{E}[\mu_N]|_B. \]
In particular, $\nu_0=\overline{\mu}_N$.
Let us use the notation $\nu_{K+1} = \mu_N$.
Since the $L^p$ norm of a metric satisfies the triangle inequality, we have that
\[ \mathbb{E}[W_p^p(\mu_N,\overline{\mu}_N)]^{1/p} \leqslant \sum_{k=0}^K \mathbb{E}
[ W_p^p(\nu_k,\nu_{k+1})]^{1/p} . \]
Indeed, the triangle inequality for $W_p$ gives us
\[W_p(\mu_N,\overline{\mu}_N) \leqslant \sum_{k=0}^{K} W_p(\nu_k,\nu_{k+1})\]
so that, taking the $L^p$ norm
$\|X\|_{L^p} = \mathbb{E}[|X|^p]^{1/p}$ at both sides
we get
\[\|W_p(\mu_N,\overline{\mu_N})\|_{L^p}
\leqslant \big\|\sum_{k=0}^{K} W_p(\nu_k,\nu_{k+1})\big\|_{L^p}
\leqslant \sum_{k=0}^K \big\| W_p(\nu_k,\nu_{k+1})\big\|_{L^p}\]
which is what we wanted.

\vspace{3mm}
\noindent
\textbf{Step 2 -- Comparison between $\nu_k$ and $\nu_{k+1}$:} We are now left to bound from above the quantity $\mathbb{E}[W_p^p(\nu_k,\nu_{k+1})]$, which will be done by controlling the transport cost on each of the squares of $\mathcal B_k$ and then gluing together the squares.
For a given level $k \in \{0,\dots,K-1\}$ and a given square $B \in \mathcal B_k$,
define the good event $G_B$ by
\[ G_{B} = \{
\mu_N(B)  \geq 0.5 \mathbb{E}[\mu_N(B)]  \} .\]
What is important from the number $0.5$ in the definition of $G_B$ is that it is in the interval $(0,1)$,
we may take whichever we prefer.
We start by writing, for $B \in \mathcal B_k$,
\begin{equation}\label{eq:decompo}
	\mathbb{E}[W_p^p(\nu_{k}|_{_{B}},\nu_{k+1}|_{_{B}})]= \mathbb{E}[W_p^p(\nu_{k}|_{_{B}},\nu_{k+1}|_{_{B}})1_{G_B}]+ \mathbb{E}[W_p^p(\nu_{k}|_{_{B}},\nu_{k+1}|_{_{B}})1_{G_B^c}]
\end{equation}
and we bound each term separately.

\vspace{3mm}
\noindent
\textbf{Substep 2.1 -- Inside the event $G_B$:} For the first term, we notice that, on the event $G_B$, the density of $\nu_{k}|_{B}$ with respect to Lebesgue measure
is bounded from below by $0.5 a$, where $a$ is given in Hypothesis
\ref{Hypo3:density_minimal}. Hence we can use
the following lemma from \cite[Lemma 3.4]{goldmantrevisan}.

\begin{lemma} \label{lem:goldmantrevisan}
	Let $R$ be a square and let $\mu$ and $\lambda$ be two measures on $R$ with equal mass, both absolutely continuous with respect to Lebesgue and such that $\inf_R \lambda >0$.
	Then for every $p\geq1$,
	\[ W_p^p(\mu,\lambda) \leqslant \theta_p\frac{\mathrm{diam}(R)^p}{(\inf_R \lambda)^{p-1}} \int_R |\mu-\lambda|^p, \]
	where the constant $\theta_p$ only depends on $p$.
\end{lemma}
Notice that the $p=1$ case can be obtained
by the
identity $W_1(\mu,\lambda)
= \sup_{\|f\|_{\mathrm{Lip}} \leq 1} \int f (\mu - \lambda)$
and does not need a lower bound for the density.
For $B \in \mathcal B_k$, there are four squares
in $\mathcal B_{k+1}$ that are contained in $B$.
Using Lemma \ref{lem:goldmantrevisan}, we can write
\begin{align*}
	\mathbb{E}[W_p^p(\nu_{k}|_{_{B}}, & \nu_{k+1}|_{_{B}}) 1_{G_{B}}] \\
	& \leqslant \theta_p(0.5 a)^{1-p} \mathrm{diam}(B)^p
	\mathbb{E} \Big[ \int_{B}
	\Big|\nu_{k}|_{_B}-\nu_{k+1}|_{_B} \Big|^p \Big] \\
	& = \theta_p(0.5 a)^{1-p} \mathrm{diam}(B)^p
	\sum_{\tilde B \in \mathcal B_{k+1},
		\tilde B \subset B}\mathbb{E} \Big[ \Big| \frac{\mu_N(B)}{\mathbb{E}[\mu_N(B)]} -
	\frac{\mu_N(\tilde B)}{\mathbb{E}[\mu_N(\tilde B)]}  \Big|^p \Big]
	\int_{\tilde B} \mathbb{E}[\mu_N](x)^p \mathrm d x,
\end{align*}
where $\mathbb{E}[\mu_N](\cdot)$ denotes the probability density function of $\mathbb{E}[\mu_N]$ with respect to Lebesgue measure.
Use \eqref{eq:Hyp4} and that $|\tilde B| = |B|/4$ to obtain
\begin{align*}
	\mathbb{E} \Big[ \Big| \frac{\mu_N(B)}{\mathbb{E}[\mu_N(B)]} -
	\frac{\mu_N(\tilde B)}{\mathbb{E}[\mu_N(\tilde B)]}  \Big|^p \Big]
	& =
	\mathbb{E} \Big[ \Big| \frac{\mu_N(B)}{\mathbb{E}[\mu_N(B)]} - 1
	+ 1-
	\frac{\mu_N(\tilde B)}{\mathbb{E}[\mu_N(\tilde B)]}  \Big|^p \Big]		\\
	& \leqslant 2^{p-1}\Big(
	\mathbb{E} \Big[ \Big| \frac{\mu_N(B)}{\mathbb{E}[\mu_N(B)]} - 1
	\Big|^p\Big] + \mathbb{E}
	\Big[\Big|
	1-
	\frac{\mu_N(\tilde B)}{\mathbb{E}[\mu_N(\tilde B)]}  \Big|^p \Big]
	\Big)	\\
	&\leqslant
	\frac{2^{p-1}}{a^p} \big(
	|B|^{-p/2} g(|B|)^p + |\tilde B|^{-p/2} g(|\tilde B|)^p \big)       \\
	&\leqslant
	\frac{2^{p-1}}{a^p} (1 + 2^p) |B|^{-p/2} g(|B|/4)^p.
\end{align*}
By defining
$\alpha_1 = \theta_p(0.5 a)^{1-p} 2^{p-1} a^{-p} (1+2^p)$, we have obtained
\[\mathbb{E}[W_p^p(\nu_{k}|_{_{B}},\nu_{k+1}|_{_{B}})
1_{G_{B}}]
\leqslant \alpha_1
\mathrm{diam}(B)^p|B|^{-p/2} g(|B|/4)^p \int_{B} \mathbb{E}[\mu_N](x)^p \mathrm d x. \]
\noindent
\textbf{Substep 2.2 -- Outside of the event $G_B$:} Let us now consider
$\mathbb{E}[W_p^p(\nu_{k}|_{_{B}},\nu_{k+1}|_{_{B}})1_{G_{B}^c}]$.
We have the crude bound
\[
\mathbb{E}[W_p^p(\nu_{k}|_{_{B}},\nu_{k+1}|_{_{B}})1_{G_{B}^c}]
\leqslant
\mathrm{diam}(B)^p\, 0.5 \mathbb{E}[\mu_N(B)]
\mathbb P(G_B^c).
\]
For $\mathbb P(G_B^c)$ we may use
Markov's inequality to get that
\[
\mathbb P(G_B^c)
=\mathbb P\Big( 1-\frac{\mu_N(B)}{\mathbb{E}[\mu_N(B)]} > 0.5  \Big)							 \leqslant
\mathbb P\Big( \Big|1-\frac{\mu_N(B)}{\mathbb{E}[\mu_N(B)]}\Big|^p > 0.5^p  \Big)														\\
\leqslant
0.5^{-p} \,\mathbb E\Big[ \Big|1-\frac{\mu_N(B)}{\mathbb{E}[\mu_N(B)]}\Big|^p
\Big].
\]
Define
$\alpha_2=(0.5)^{1-p} a^{-p}$.
So, by using that
$g$ is non-increasing
and \eqref{eq:Hyp4}, we have obtained
\[\mathbb{E}[W_p^p(\nu_{k}|_{_{B}},\nu_{k+1}|_{_{B}})1_{G_{B}^c}]
\leqslant
\alpha_2 \mathrm{diam}(B)^p |B|^{-p/2} g(|B|/4)^p  \mathbb{E}[\mu_N(B)].\]
\noindent
\textbf{Step 3 -- Gluing the squares:}
Notice that for $B\in \mathcal B_k$ we have
$\mathrm{diam}(B)^p|B|^{-p/2} = 2^{p/2}$
and $|B| = N/4^k$.
By using that $\sum_{B \in \mathcal B_k}\nu_{k}|_{_{B}}
= \nu_k$ and
$\sum_{B \in \mathcal B_k}\nu_{k+1}|_{_{B}}
= \nu_{k+1}$, we have
\[\mathbb{E}[W_p^p(\nu_{k},\nu_{k+1})]
\leq
2^{p/2} g\bigg(\frac{N}{4^{k+1}}\bigg)^p
\bigg(
\alpha_1 \int_{Q_N} \mathbb E[\mu_N](x)^p \mathrm d x
+
\alpha_2 \mathbb E[\mu_N(Q_N)] \bigg)\]
for $k + 1 \leq K $. Notice that, by the definition of $K$, we have that
$\frac{N}{4}\leqslant 4^K \leqslant N$. Then,
\begin{align*}
	\mathbb E[W_p^p(\nu_0,\nu_K )]^{1/p} &\leqslant
	\sum_{k=0}^{K-1}\mathbb E[W_p^p(\nu_k,\nu_{k+1} )]^{1/p}            \\
	&\leqslant \sqrt 2
	\bigg(
	\alpha_1 \int_{Q_N} \mathbb E[\mu_N](x)^p \mathrm d x
	+
	\alpha_2 \mathbb E[\mu_N(Q_N)] \bigg)^{1/p}
	\sum_{k=0}^{K-1} g\bigg(\frac{N}{4^{k+1}}\bigg)                     \\
	&\leq \sqrt 2
	\bigg(
	\alpha_1 \int_{Q_N} \mathbb E[\mu_N](x)^p \mathrm d x
	+
	\alpha_2 \mathbb E[\mu_N(Q_N)] \bigg)^{1/p}
	\Bigg(\int_{0}^{K-1}
	g\bigg(\frac{N}{4^{t+1}}\bigg) \mathrm d t
	+ g(1)
	\Bigg),
\end{align*}
as $g$ is assumed to be non increasing.
By a change of variables we obtain
\[\int_{0}^{K-1}
g\bigg(\frac{N}{4^{t+1}}\bigg) \mathrm d t
= \frac{1}{\ln 4}\int_{\frac{N}{4^{K}}}^{N}
g(y) \frac{\mathrm d y}{y}
\leq
\frac{1}{\ln 4}
\int_1^N
g(y) \frac{\mathrm d y}{y}.\]
We have one term left to understand, namely
$\mathbb E[W_p^p(\nu_K,\mu_N)]$. We can use that
\[W_p^p(\nu_K,\mu_N) \leqslant \sum_{B \in \mathcal B_{K}}
W_p^p(\nu_K|_{_B}, \mu_N|_{_B}) \leqslant
\sum_{B \in \mathcal B_{K}}  \mathrm{diam}(B)^p \mu_N(B)
\]
We now use that the diameter of each $B \in \mathcal B_K$
is
$\frac{\sqrt{2N}}{2^K} =  \sqrt{\frac{2N}{4^K}} \leqslant 2 \sqrt 2$
and that
$ \sum_{B \in \mathcal B_{K}}   \mu_N(B)$ is $\mu_N(Q_N)$.
Here we are using that $\mu_N (\partial Q_N) = 0$,
consequence of $\mathbb{E}[\mu_N(\partial Q_N)] = 0$.
Then,
\[\mathbb{E}[W_p^p(\nu_K,\mu_N)] \leqslant
2^{3p/2} \mathbb{E}[\mu_N(Q_N)].\]
We conclude the inequality
\begin{align*}
	\mathbb E[W_p^p(\mu_N,\overline{\mu}_N)]^{1/p} \leqslant &
	\sqrt 2
	\bigg(
	\alpha_1 \int_{Q_N} \mathbb E[\mu_N](x)^p \mathrm d x
	+
	\alpha_2 \mathbb E[\mu_N(Q_N)] \bigg)^{1/p}
	\bigg(g(1)  + \int_1^N
	g(y) \frac{\mathrm d y}{y} \bigg)              \\
	& + 2^{3/2} \mathbb{E}[\mu_N(Q_N)]^{1/p}.
\end{align*}
By using that $\mathbb E[\mu_N](x) \leq A$
and defining
$\alpha = 2^{p-1} \sqrt 2^{p}
(\alpha_1 A^p  + (\alpha_2 + 2^{3p/2}) A)$ we get
\[\mathbb E[W_p^p(\mu_N,\overline{\mu}_N)]
\leq \alpha N
\bigg(1 + g(1) + \int_1^N g(y) \frac{\mathrm d y}{y} \bigg)^p
\leq
\alpha N
\bigg(1 + g(1) + g(N) + \int_1^N g(y) \frac{\mathrm d y}{y} \bigg)^p.\]
Since $A \leq A^p$ and $a^{-p} \leq a^{1-2p}$,
we can bound $\alpha$ by a $a^{1-2p} A^p$ times a constant
$C_p$ that only depends on $p$ (for instance,
$C_p = 2^{5p} (1+\theta_p)$ works).

\subsection{Proof of Theorem \ref{thm:minimal-hypotheses dimension}}
The proof of Theorem \ref{thm:minimal-hypotheses dimension} follows the proof of Theorem \ref{thm:minimal-hypotheses} line by line. One has to divide the cube $[0,N^{1/d}]^d$ in $2^{dk}$ sub-cubes of size $2^{-k}N^{1/d}$. The key
\cite[Lemma 3.4]{goldmantrevisan} (stated here as  Lemma \ref{lem:goldmantrevisan}) is valid in any dimension which allows us to generalize the proof.

\subsection{Proof of Corollary \ref{cor:main}}

Theorem \ref{thm:minimal-hypotheses}
gives the upper bound if we define $g(x) = \gamma x^{-\varepsilon/2}$,
\[\mu_N = X|_{C_N} \quad \mbox{ and }
\overline{\mu}_N =
\frac{X(C_N)}{\mathbb{E}[X(C_N)]} \mathbb{E}[X]|_{C_N}.\]
For the lower bound, we define $\Omega=\{\frac{a}{2}N \leq X(C_N) \leq 2AN  \} $ so that, on $\Omega$, the density of $\overline{\mu_N}$ is bounded from above by $2A^2/a$. Hence, by Lemma \ref{lem:lowerbound}, we get
\[W_p^p(\mu_N, \overline{\mu}_N)
\geq  \frac{\alpha_2^pa^{p/2}}{(2A^2)^{p/2}} X(C_N) 1_{\Omega} \geq c N 1_{\Omega},\]
where $c=\frac{\alpha_2^pa^{p/2}}{(2A^2)^{p/2}} \frac{a}{2}>0$. Then $\mathbb{E}[W_p^p (\mu_N, \overline{\mu}_N)]
\geq
c N \mathbb P(\Omega)$.
If we show that $\mathbb{P}(\Omega)>0$ then the proof is complete.
\begin{align*}
 \mathbb{P}(\Omega^c) & = \mathbb{P}(X(C_N)<\frac{a}{2}N)+\mathbb{P}(X(C_N)>2AN)\\
 & \leq \mathbb{P}(\mathbb{E}(X(C_N))-X(C_N) > \mathbb{E}(X(C_N))-\frac{a}{2}N) + \mathbb{P}(X(C_N)>2AN)\\
 & \leq \mathbb{P}(\mathbb{E}(X(C_N))-X(C_N) > \frac{a}{2}N) + \frac{AN}{2AN},
\end{align*}
by Markov inequality and Hypothesis \ref{Hypo3:density_main} of Corollary \ref{cor:main}. Hence
\begin{align*}
 \mathbb{P}(\Omega^c) & \leq \mathbb{P}(|X(C_N)-\mathbb{E}(X(C_N))| > \frac{a}{2}N) + \frac{1}{2}\\
 & \leq \frac{1}{2}+\frac{\gamma N^{(1-\varepsilon)p/2}}{(aN/2)^p}\\
 & \leq \frac{1}{2}+\frac{\gamma}{(a/2)^p} N^{-(1+\varepsilon)p/2},
\end{align*}
making use of Hypothesis \ref{Hypo1:moments_main}. Choosing $N\geq N_0$ large enough yields $ \mathbb{P}(\Omega^c)<1$. This concludes the proof. To extend this result to $1\leq N \leq N_0$, we can use the continuity and positivity of $N \mapsto \mathbb{E}(W_p^p(\mu_N,\overline{\mu_N}))/N$ which is a consequence of Hypothesis \ref{Hypo3:density_main}.

\subsection{Proof of Proposition \ref{prop:DPP}}
Consider a Hermitian determinantal point process $X$ on $\R^d$
(or any space). Then, for any measurable $B \subset \R^d$, we have $X(B)=\sum_{i=1}^n\xi_i$, where $n \in \N \cup \{+\infty\}$ and the $\xi_i$'s are independent Bernoulli variables
\cite[Theorem 4.5.3]{DPPs}.
Bernstein's inequality
\cite[Theorem 3]{BLB04}
applied to the bounded centered independent random variables $\xi_i-\mathbb{E}[\xi_i]$ yields
\[\P(|X(B)-\mathbb{E}[X(B)]| \geqslant t) \leqslant 2\exp\Big(-\frac{t^2}{2\Var(X(B))+t/3}\Big)\]
for every $t >0$. Then,
\begin{align*}
	\mathbb{E}[|X(B) & - \mathbb{E}[X(B)]|^p]\\
	& = \int_0^{+\infty}pt^{p-1}\P(|X(B) - \mathbb{E}[X(B)]|>t)
	\mathrm dt\\
	& \leqslant 2\int_0^{+\infty}pt^{p-1}e^{-t^2/(2\Var(X(B))+t/3)}
	\mathrm dt\\
	& \leqslant 2\int_0^{6\Var(X(B))}pt^{p-1}e^{-t^2/(4\Var(X(B)))}
	\mathrm dt+
	2\int_{6\Var(X(B))}^{+\infty}pt^{p-1}e^{-3t/2} \mathrm dt\\
	& \leqslant 2^{p+1}\Var(X(B))^{p/2}
	\int_0^{3\sqrt{\Var(X(B))}}pu^{p-1    }e^{-u^2} \mathrm dt\\
	& \hspace{5cm}
	+2e^{-3 \Var(X(B))/4}\int_{6\Var(X(B))}^{+\infty}pt^{p-1}e^{-3t/4} \mathrm dt\\
	& \leqslant 2^{p+1}\Var(X(B))^{p/2}\int_0^{+\infty}pu^{p-1}e^{-u^2}
	\mathrm dt+2e^{-3 \Var(X(B))/4}\int_0^{+\infty}pt^{p-1}e^{-3t/4}
	\mathrm dt\\
	& \leqslant \gamma_p \big(\Var(X(B))^{p/2}+e^{-3 \Var(X(B))/4}\big)
	\\
	& \leqslant \gamma_p (\Var(X(B))^{p/2}+ 1 )
\end{align*}
where $\gamma_p$ is the maximum of both integrals
times $2^{p+1}$.

\section{Proof of Proposition \ref{Ginibre}}\label{sec:ProofGinibre}
The main tool to prove Proposition \ref{Ginibre} is the Wasserstein distance between point processes, which was introduced by \cite{Huesmann2016}. This distance should not be mistaken with the classical Wasserstein distance $W_p$ which we introduced earlier. We refer to the recent survey \cite{lachiezereySurvey} for a quick presentation of this distance, or to \cite{EHJM} for a detailed study of it. This Wasserstein distance on point processes is only defined for translation invariant point processes with unit intensity, and for two such processes $X$ and $Y$, their distance will be written as $\mathrm{Wass}_p(X,Y)$.
We can define $\mathrm{Wass}_p$ on the space of stationary random measures with unit intensity as
\[\mathrm{Wass}_p^p (X, Y) = \inf_{Q \in \mathrm{cpl}(X,Y)} \mathbb E_{Q}\Big[ \inf_{q \in \mathrm{cpl}(X^\omega,Y^\omega)} \limsup_{N\to \infty} \frac{1}{N} \int_{\Lambda_N \times \mathbb{R}^d} |x-y|^d \mathrm d q(x,y) \Big]\]
where $\Lambda_N = [-N^{1/d}/2,N^{1/d}/2]^d$.
The original random measures $X$ and $Y$
may be defined on different $\Omega$'s
but once the coupling $Q$ is fixed,
$X$ and $Y$ denote the new coupled random measures,
and $X^\omega$ and $Y^\omega$ their realizations.
Notice that, strictly speaking,
$\mathrm{Wass}_p(X, Y)$ is a distance
on the space of distributions of translation invariant random measures with unit intensity
so that we will denote
it by $\mathrm{Wass}_p (P_1, P_2)$
for $P_1$ and $P_2$ the laws of $X$ and $Y$
respectively.

\medskip
\noindent
We will use \cite[Propositions 2.10 and 2.11]{EHJM}
which we reproduce here for completeness. Here
we consider the action of $\mathbb R^2$ on
$\mathbb R^2 \times \mathbb R^2$ given by
translations, i.e., a vector $v \in \mathbb R^2$
acts on an element $(x,y) \in \mathbb R^2 \times \mathbb R^2$
by taking it to $(x + v, y+ v) \in \mathbb R^2 \times
\mathbb R^2$.

\begin{proposition}[\cite{EHJM} Propositions 2.10 and 2.11]
	$\mathrm{Wass}_p$ is lower semi-continuous for the
	topology of weak convergence on
	the product of the spaces of probability measures. Moreover, if $P_1$ and $P_2$ are the distributions of stationary point processes satisfying $\mathrm{Wass}_p(P_1,P_2) < +\infty$,
	then there exists a point process $q$
	on $\mathbb R^2 \times \mathbb R^2$
	invariant under translations by
	$\mathbb R^2$,
	whose (random) marginals
	$X$ and $Y$ are distributed according to $P_1$ and $P_2$
	and such that
	\[ \mathrm{Wass}_p^p(P_1,P_2) = \mathbb{E}
	\Big( \int_{\Lambda_1 \times \mathbb{R}} |x-y|^p q(\mathrm dx,\mathrm dy)\Big) . \]
	This point process $q$ is called an optimal matching.
\end{proposition}
First, let us explain how this proposition will imply Proposition \ref{Ginibre}. The stationary random lattice is  $Z=\{ k +U, k \in \mathbb{Z}^2 \}$, where $U$ is a uniform random variable on $[0,1]^d$. The distribution of $Z$ is written as $P_Z$. If we show that $\mathrm{Wass}_p(\mathrm{Gin},P_Z)$ is finite, then there exists a  (random) coupling between $\mathrm{Gin}$ and $P_Z$ for which each realization is a matching. This matching gives the (stationary) displacement field from the perturbed lattice $Z$ to the points of the Ginibre ensemble. Since $\mathrm{Wass}_p(P_Z, \mathrm{Leb})$ is finite, it is sufficient to prove that $\mathrm{Wass}_p(\mathrm{Gin}, \mathrm{Leb})$ is finite.

\medskip
\noindent
This will follow from the construction of two auxiliary random measures $X_N$ and $Y_N$, the first one converging towards $\mathrm{Gin}$ and the second one converging towards the Lebesgue measure, which satisfy $\mathrm{Wass}_p(X_N,Y_N) \leq C_p$. This construction is analogous to the one presented in the proof of \cite[section 2.3, p17]{huesmannleble}.

\medskip
\noindent
For $\Lambda_N=[-\sqrt{N}/2,\sqrt{N}/2]^2$, we define
\[ \hat{X}_N = \bigcup_{k \in \mathbb{Z}^2} \Big( \big(\mathrm{Gin}\cap \Lambda_N \big) + \sqrt{N}k
\Big).\]
This point process is the $\sqrt{N}\mathbb{Z}^2$ periodic version of the restriction of Ginibre to $\Lambda_N$. Then, if $U$ is a uniform random variable on $[-1/2,1/2]^2$
independent of $\mathrm{Gin}$, we set $X_N = \hat{X}_N + \sqrt{N}U$. We say that $X_N$ is the periodic and stationary version of $\mathrm{Gin}\cap \Lambda_N$. By definition, $X_N$ is a stationary random measure, and its intensity is equal to $1$. In addition, we use the same uniform random variable $U$ for all $N$. This coupling will help later.

\medskip
\noindent
On the other hand, making the random measure $\frac{\mathrm{Gin}(\Lambda_N)}{N} \mathrm{Leb}_{|\Lambda_N}$ periodic gives the measure $\frac{\mathrm{Gin}(\Lambda_N)}{N} \mathrm{Leb}$, which is already translation invariant.
We immediately get that $X_N$ and $\frac{\mathrm{Gin}(\Lambda_N)}{N} \mathrm{Leb}$ have unit intensity.

\medskip
\noindent
Applying Theorem \ref{thm:minimal-hypotheses} to the Ginibre point process, there exists a (random) coupling $\Pi_N$ between $\mathrm{Gin} \cap \Lambda_N$ and $\frac{\mathrm{ Gin}(\Lambda_N)}{N} \mathrm{Leb}_{|\Lambda_N}$ such that
\[\mathbb{E}\big[W_p^p(\mathrm{Gin} \cap \Lambda_N,\frac{\mathrm{Gin}(\Lambda_N)}{N} \mathrm{Leb}_{|\Lambda_N} )\big] = \mathbb{E}\Big[\int |x-y|^p d\Pi_N(x,y)\Big] \leq C_p N.\]
By copying and stationarizing $\Pi_N$, one gets a translation invariant coupling $\Pi^*_N$ between $X_N$ and $\frac{\mathrm{Gin}(\Lambda_N)}{N} \mathrm{Leb}$ with cost bounded from above by $C_p$, i.e.,
\[ \mathrm{Wass}_p^p\Big(X_N, \frac{\mathrm{Gin}(\Lambda_N)}{N} \mathrm{Leb}\Big) \leq C_p  .\]
Finally, since we took the same uniform random variable for all $N$ in the definition of $X_N$ and, for any fixed $\omega$,  $\Lambda_1 + U(\omega)$ contains a neighborhood of the origin, we get that for any compact $K \subset \mathbb{C}$ there exists $N_{\omega}$ such that $\forall N \geq N_{\omega}, K \subset \Lambda_N+\sqrt{N}U(\omega)$. Hence, for large $N$, a compactly supported function will only see the points in the central block, implying the weak convergence of $X_N$ to the Ginibre point process.
The weak convergence of $\frac{\mathrm{Gin}(\Lambda_N)}{N} \mathrm{Leb}$ comes from the convergence in probability of $\frac{\mathrm{Gin}(\Lambda_N)}{N}$ towards $1$. This ends the proof of Proposition \ref{Ginibre}.

\section{Appendix: Proof of the lower bound}\label{sec:Appendix}

\begin{proof}[Proof of Lemma \ref{lem:lowerbound}]
	Let us begin with $p=1$.
	Define the 1-Lipschitz function
	$f: \mathbb R^d \to [0,\infty[$,
	\[f(x) = \max \big\{(c-|x-x_1|)_+,\dots,(c-|x-x_N|)_+\big\},\]
	so that $f(x_i) = c$.
	Define $\mu_N = \sum_{i=1}^N \delta_{x_i}$ and notice that
	$\int_{\mathbb R^d} f \mathrm d \mu_N = c N$.
	Since $\mu$ has a density bounded from above by $A$,
	\[\int_{\mathbb R^d} f \mathrm d \mu
	\leq  A \sum_{i=1}^N
	\int_{\mathbb R^d} 1_{\{|x-x_i|\leq c\}} (c-|x-x_i|) \mathrm d x .\]
	Denote by $B_1$ the unit ball in $\R^d$. Since
	\[\int_{\mathbb R^d} 1_{\{|x-y|\leq c\}} (c-|x-y|) \mathrm dx
	=c |B_1| c^d  - |\partial B_1|
	\int_0^{c} r^d \mathrm d r
	=\frac{c^{d+1} |\partial B_1|}{d(d+1)} ,\]
	we get that
	\[\int_{\mathbb R^d}f \mathrm d \mu_N
	-\int_{\mathbb R^d}f \mathrm d \mu
	\geq N
	\left( c - A\frac{c^{d+1} |\partial B_1|}{d(d+1)}  \right). \]
	Taking $c$ that maximizes
	$c \mapsto c- A\frac{c^{d+1} |\partial B_1|}{d(d+1)} $,
	i.e., $c =\big(\frac{d}{A|\partial B_1|}\big)^{1/d} $ gives
	\[W_1(\mu_N,\mu)
	\geq N\Big(\frac{d}{A|\partial B_1|}\Big)^{1/d} \Big(\frac{d}{d+1}\Big) .\]
	By Hölder's inequality,
	$W_1(\mu,\nu) \leq N^{1-\frac{1}{p}} W_p(\mu,\nu)$,
	so that
	\[W_p(\mu_N,\mu)
	\geq N^{1/p}
	\Big(\frac{d}{A|\partial B_1|}\Big)^{1/d} \Big(\frac{d}{d+1}\Big) .\]
\end{proof}


\end{document}